\documentclass{article}

\usepackage{amssymb}
\usepackage{amsmath}
\usepackage{amsthm}
\newtheorem{theorem}{Theorem}[section]
\newtheorem{lemma}[theorem]{Lemma}

\theoremstyle{definition}

\newtheorem{example}[theorem]{Example}
\newtheorem{proposition}[theorem]{Proposition}

\theoremstyle{remark}

\numberwithin{equation}{section}

   \newcommand{\ph}{\mbox{$\varphi$}}
    
    \renewcommand{\phi}{\varphi}
   \newcommand{\Ht}{\mbox{$H^{2}$}}
   
   \newcommand{\Hi}{\mbox{$H^{\infty}$}}
   \newcommand{\D}{\mbox{$\mathbb{D}$}}
   \newcommand{\C}{\mbox{$C_{\varphi}$}}

   \newcommand{\W}{\mbox{$W_{\psi,\varphi}$}}
   \newcommand{\Ws}{\mbox{$W_{\psi,\varphi}^{\textstyle\ast}$}}
 	\newfont{\caps}{cmcsc9}  
 	\newfont{\jour}{cmti9}  
\begin{document}
\title{Quasinormal and hyponormal weighted composition operators on $H^2$ and $A^2_{\alpha}$ with linear fractional compositional symbol}
\author{Mahsa Fatehi, Mahmood Haji Shaabani, Derek Thompson}
\maketitle

\begin{abstract}
In this paper, we study quasinormal and hyponormal composition operators \W with linear fractional compositional symbol $\ph$ on the Hardy and weighted Bergman spaces. We characterize the quasinormal composition operators induced on $H^{2}$ and $A_{\alpha}^{2}$ by these maps and many such weighted composition operators, showing that they are necessarily normal in all known cases. We eliminate several possibilities for hyponormal weighted composition operators but also give new examples of hyponormal weighted composition operators on $H^2$ which are not quasinormal.
\end{abstract}

 \section{Introduction}
\subsection{History.} 
An operator $A$ on a Hilbert space $H$ is \textit{hyponormal} if $A^{\ast}A-AA^{\ast} \geq 0$, or equivalently if $\|A^{\ast}f\| \leq \|Af\|$ for all $f \in H$. An operator $A$ is \textit{quasinormal} if $A$ and $A^{\ast}A$ commute.  It is known that every quasinormal operator is hyponormal (see \cite[Proposition 1.7 p. 29]{c4} and \cite[Proposition 4.2, p. 46]{c4}), and certainly every normal operator is both quasinormal and hyponormal. 

Normality and similar weaker conditions have been intensely studied in recent years for weighted composition operators on $H^2$ and related spaces. The normal and unitary weighted composition operators on $H^{2}$ were discovered in part by Bourdon and Narayan \cite{Bourdon}. They showed that every automorphism $\varphi$ of $\mathbb{D}$ has a companion weight function $\psi$ such that \W\ is unitary on $H^{2}$. They also found some other normal weighted composition operators on $H^{2}$. Le \cite{l} found similar results in several variables. Recently in \cite{coko}, Cowen, Jung, and Ko investigated when \Ws\ is hyponormal, and Jung, Kim, and Ko \cite{jkk}, while studying binormal composition operators, showed that \C\ on $H^2$ is quasinormal if and only it is normal.

We are focused on discovering hyponormal and quasinormal weighted composition operators on $H^{2}$ and $A_{\alpha}^{2}$ with linear fractional compositional symbol. The rest of Section 1 will give necessary preliminaries, including a proof of the spectrum of $C_{\ph}$ on $A^2_{\alpha}$ when \ph\ is a parabolic non-automorphism. In Section 2, we show that for almost all cases, if $\ph \in \mbox{LFT}(\D)$ and $\psi \in \Hi$ is continuous at the Denjoy-Wolff point of $\ph$, then \W\ is quasinormal if and only if it is normal (the only possible exception is when \ph\ is an automorphism but \W\ is not invertible). We also show that \C\ is quasinormal with $\ph \in \mbox{LFT}(\D)$ if and only if it is normal and $\ph(z) = \lambda z$ where $|\lambda|\leq 1$ (this extends the theorem of \cite{jkk} to $A^2_{\alpha}$). In Section 3, we turn our attention to hyponormal weighted composition operators. In particular, we give examples of new hyponormal weighted composition operators on \Ht\ which are not quasinormal, and eliminate some other possibilities.

\subsection{Preliminaries.} 
The Hilbert spaces we are considering are the classical Hardy space $H^{2}$ and the weighted Bergman spaces $A_{\alpha}^{2}$ on the complex unit disk $\D$. For $f = \sum_{n=0}^{\infty} a_n z^{n}$ analytic on $\mathbb{D}$, $H^2$ is the set 
$$\{ f: \|f\|^{2}=\sum_{n=0}^{\infty}|a_n|^{2}<\infty \}.$$
For $\alpha > -1$, the weighted Bergman space $A_{\alpha}^{2}$ consists of all analytic functions $f$ on $\mathbb{D}$ such that
$$\|f\|_{\alpha+2}^{2}=\int_{\mathbb{D}}|f(z)|^{2}(\alpha+1)(1-|z|^{2})^{\alpha}dA(z)<\infty,$$
where $dA$ is the normalized area measure on $\mathbb{D}$. The case when $\alpha=0$ is the (unweighted) Bergman space, denoted $A^{2}$. Both the weighted Bergman space and the Hardy space are reproducing kernel Hilbert spaces, when the reproducing kernel for evaluation at $w$ is given by $K_{w}(z)=(1-\overline{w}z)^{-\gamma}$ for $z,w \in \mathbb{D}$, with $\gamma=1$ for $H^{2}$ and $\gamma=\alpha+2$ for $A_{\alpha}^{2}$. We write $H^{\infty}$ for the space of bounded analytic functions on $\mathbb{D}$, and denote its norm by $\|.\|_{\infty}$, i.e.
$$\|f\|_{\infty}:=\sup_{z\in \mathbb{D}}|f(z)|.$$
A \textit{composition operator} $C_{\varphi}$ on $H^2$ or $A^2_{\alpha}$
 is defined by the rule $C_{\ph}(f)=f \circ \ph$ for $\ph: \D \rightarrow \D$ and analytic in $\D$. Moreover, for $\psi \in \Hi$ and an analytic self-map \ph\ of $\mathbb{D}$, we define the weighted composition operator \W\ by $\W f=\psi f \circ \varphi$ . Such weighted composition operators are clearly bounded on $H^{2}$ and $A_{\alpha}^{2}$.

A linear fractional self-map of $\mathbb{D}$ is a map of the form $\varphi(z)=(az+b)/(cz+d), ad-bc \neq 0$ for which $\varphi(\mathbb{D}) \subseteq \mathbb{D}$. We denote the set of those maps by $\mbox{LFT}(\mathbb{D})$. If $\varphi$,
not the identity map, is considered as a map of the extended complex plane $\mathbb{C} \cup \{\infty\}$ onto itself, it will have exactly two fixed points, counting multiplicities. The automorphisms of $\mathbb{D}$, denoted $\mbox{Aut}(\mathbb{D})$, are the maps in $\mbox{LFT}(\mathbb{D})$ that take $\mathbb{D}$
onto itself. They are necessarily of the form $\varphi(z)=\lambda(a-z)/(1-\overline{a}z)$, where $|\lambda|=1$ and $|a| < 1$ (see \cite{c3}). The automorphisms are divided into three subclasses, based on their fixed point behavior:

\begin{enumerate}
\item \textit{elliptic} if \ph\ has an interior fixed point,
\item \textit{hyperbolic} if \ph\ has a boundary fixed point $\zeta$ with $\ph'(\zeta) < 1$, and
\item \textit{parabolic} if \ph\ has a boundary fixed point $\zeta$ with $\ph'(\zeta) = 1$.  
\end{enumerate}

When describing the derivative at the boundary, we mean this in the sense of radial limits, since \ph\ need only be defined on \D. However, since the automorphisms (and all linear fractional-maps) extend to maps of $\mathbb{C} \cup \{ \infty \}$, there is no confusion here.

Linear fractional transformations are particularly interesting choices for the symbol of a composition operator, because of their connection to the reproducing kernels. This is exemplified in the Cowen adjoint formula. Though well-known, we re-state here due to its repeated use.

\begin{proposition}[Cowen adjoint formula]\label{cowen} Suppose $\ph= (az+b)/(cz+d)$ maps \D\ into \D. Then the adjoint of $C_{\varphi}$  acting on $H^{2}$ and $A^{2}_{\alpha}$ is given by
$$C_{\varphi}^{\ast}=T_{g}C_{\sigma}T_{h}^{\ast},$$
where
\begin{enumerate}
\item $\sigma(z):=({\overline{a}z-\overline{c}})/({-\overline{b}z+\overline{d}})$
is a self-map of $\mathbb{D}$, 
\item $g(z):=(-\overline{b}z+\overline{d})^{-\gamma}$ and $h(z):=(cz+d)^{\gamma}$, with $\gamma=1$ for $H^{2}$ and $\gamma=\alpha+2$ for $A^{2}_{\alpha}$, and
\item $g$ and $h$ belong to $H^{\infty}$.
\end{enumerate}
The map $\sigma$ is called the \textit{Krein adjoint} of \ph. We will refer to $g$ and $h$ as the \textit{Cowen auxiliary functions} for \ph. \end{proposition}

Composition operators on these spaces are also deeply intertwined with the function theory of \D, particularly the Denjoy-Wolff theorem:

\begin{proposition}[Denjoy-Wolff] Let \ph\ be an analytic map from the open unit disk into itself which is not the identity map and not an elliptic automorphism. Then \ph\ has a unique point $\zeta$ in $\overline{\D}$ so that the iterates $\{ \ph_n \}$ of $\ph$ converge uniformly on compact subsets of \D\ to $\zeta$. \end{proposition}

Another consequence of the Denjoy-Wolff theorem is that this unique point $\zeta$ always has $\ph'(\zeta) \leq 1$.  For more on geometric function theory of the disk, see \cite{co2}.

\subsection{Parabolic non-automorphisms and a spectral interlude.} We will quickly see in Section 2 that parabolic non-automorphic maps are central to this paper, so considerable time is spent on them in this section, including finding the spectrum of composition operators with such symbols on $A^2_{\alpha}$, a fact we require later.

A map $\varphi \in \mbox{LFT}(\mathbb{D})$ is called parabolic if it has a fixed point $\zeta \in \partial \mathbb{D}$ of multiplicity $2$. The map $\tau(z):=(1+\overline{\zeta}z)/(1-\overline{\zeta}z)$ takes the unit disk onto the right half-plane $\Pi$ and sends $\zeta$ to $\infty$. Therefore, $\phi:=\tau\circ\varphi\circ\tau^{-1}$ is a self-map of $\Pi$ which fixes only $\infty$, and so must be the mapping of translation by some complex number $t$, where necessarily $\textrm{Re } t \geq 0$. Appropriately, we will call this the \textit{translation number $t$ of \ph}. Hence $\varphi(z)=\tau^{-1}(\tau(z)+t)$ for each $z \in \mathbb{D}$. Therefore, it is easy to see that
\begin{eqnarray*}
\ph(z)=\frac{(2-t)z+t\zeta}{2+t-t\overline{\zeta}z}.
\end{eqnarray*}
Note that if $\mbox{Re} \textrm{ } t=0$, then $\varphi \in \mbox{Aut}(\mathbb{D})$. If, on the other hand, $\mbox{Re}\textrm{ } t > 0$, then $\varphi  \not\in  \mbox{Aut}(\mathbb{D})$. When the translation number $t$ is strictly positive, we call $\varphi$ a \textit{positive parabolic non-automorphism}. Among the linear fractional self-maps of $\mathbb{D}$ fixing $\zeta \in \partial \mathbb{D}$, the parabolic ones are characterized by $\varphi'(\zeta)=1$.\\

Recall that the Denjoy-Wolff theorem only guarantees uniform convergence under iteration on \textit{compact subsets} of \D. In the following proposition, we can see that there are some linear fractional self-maps of $\mathbb{D}$ such that the iterates $\{ \varphi_{n} \}$  converge uniformly on \textit{all} of  $\mathbb{D}$. We will use the following result, which was proved in \cite[Example 5]{derek}, in the proof of Theorem \ref{qnn}. 

\begin{proposition}\label{uci} If \ph\ is a parabolic non-automorphism, then \ph\ converges uniformly under iteration on all of \D\ to the constant function equal to its Denjoy-Wolff point. \end{proposition}

To our knowledge, no one has found the spectrum $\sigma(C_{\ph})$ on $A^2_{\alpha}$ when \ph\ is a parabolic non-automorphism, though Cowen and MacCluer \cite[Section 7.7]{cm1} proved that $\sigma(C_{\ph})$ is a logarithmic spiral from 1 to 0 when $C_{\ph}$ acts on \Ht. Here, we prove the same fact for $A^2_{\alpha}$ by a small adjustment of the proofs found in \cite{cm1}.

\begin{proposition}\label{spectrum} Let \ph\ be a parabolic non-automorphism with Denjoy-Wolff point $\zeta$  and translation number $t$. On $H^2$ or $A^2_{\alpha}$, $$\sigma(C_{\ph}) = \{ e^{-\beta t} : 0 \leq \beta < \infty \} \cup \{ 0 \} $$ and every point of $\sigma(C_{\ph})$ except $0$ is an eigenvalue. \end{proposition}
\begin{proof} We need only prove this fact for $A^2_{\alpha}$. Without loss of generality, assume $\zeta = 1$ (otherwise, choose $\theta$ so that $e^{i\theta}=\zeta$ and consider the operator $C_{e^{i\theta}}C_{\varphi}C_{e^{-i\theta}}=C_{\tilde{\varphi}}$ whose symbol has fixed point $1$).  In \cite[Theorem 6]{hu}, Cowen and MacCluer proved $\sigma(C_{\ph}) = \{ e^{-\beta t} : |\arg \beta| = 0 \} \cup \{ 0 \}$ for $H^2$. The only time they used the fact that they were on $H^2$ was that the spectral radius of $C_{\ph}$ on $H^2$ is 1. By \cite[Theorem 6]{hu}, since $\ph'(\zeta) = 1$, the spectral radius of \C\ is also $1$ on $A^2_{\alpha}$, so the same containment holds on $A^2_{\alpha}$. In \cite[Corollary 7.42]{cm1}, the reverse containment was given by showing that each such $e^{-\beta t}$ is an eigenvalue with an eigenvector in \Hi, meaning that they belong to $A^2_{\alpha}$ as well as \Ht. Therefore, we have the desired conclusion. \end{proof}

\section{Quasinormal weighted composition operators on $H^{2}$ and $A^{2}_{\alpha}$}
In this section, we investigate quasinormal composition operators and weighted composition operators on $H^{2}$ and $A^{2}_{\alpha}$. We will look at the non-automorphic and automorphic cases separately. \\

\subsection{$\ph \notin \mbox{Aut}(\mathbb{D})$.}The next theorem explains why we have focused on the properties of parabolic non-automorphisms.

\begin{proposition}\label{parabolic} Suppose that $\varphi \in \mbox{LFT}(\mathbb{D})$ is not an automorphism of
$\mathbb{D}$ and that $\varphi(\zeta)=\eta$ for some $\zeta,\eta \in \partial\mathbb{D}$. Let $\psi \in H^{\infty}$ be continuous at $\zeta$ and $\psi(\zeta)\neq 0$. If \W\ is quasinormal on $H^{2}$ or $A^{2}_{\alpha}$, then $\varphi$ is parabolic.\end{proposition}

\begin{proof} In this proof, we will let $A \equiv B$ denote that two operators $A$ and $B$ have compact difference. Let \W\ be quasinormal on $H^{2}$ or $A^{2}_{\alpha}$. By \cite[Proposition 2.3]{fash} and \cite[Corollary 2.2]{kmm}, we obtain $\W \Ws \W \equiv \psi^{2}(\zeta)\overline{\psi(\zeta)}C_{\varphi}C_{\varphi}^{\ast}C_{\varphi}$ and $W^{\ast}_{\psi,\varphi}W_{\psi,\varphi}W_{\psi,\varphi}\equiv\psi^{2}(\zeta)\overline{\psi(\zeta)} C^{\ast}_{\varphi}C_{\varphi}C_{\varphi}$. Now using  \cite[Theorem 3.1]{kmm} and \cite[Theorem 3.2]{mw2}, we see that
\begin{eqnarray*}
  \W \Ws \W &\equiv& s \psi^{2}(\zeta)\overline{\psi(\zeta)}C_{\varphi}C_{\sigma}C_{\varphi}\\
  &=&s\psi^{2}(\zeta)\overline{\psi(\zeta)}C_{\varphi\circ\sigma\circ\varphi},
  \end{eqnarray*}
 and also by the same argument, we conclude that
\begin{eqnarray*}
  \Ws \W \W
  &\equiv&s\psi^{2}(\zeta)\overline{\psi(\zeta)}C_{\sigma}C_{\varphi}C_{\varphi}\\
  &=&s\psi^{2}(\zeta)\overline{\psi(\zeta)}C_{\varphi\circ\varphi\circ\sigma},
  \end{eqnarray*}
where $\sigma$ is the Krein adjoint of  $\varphi$ and by \cite[Proposition 3.6]{kmm} and \cite[Theorem 3.2]{mw2}, $s=|\varphi'(\zeta)|^{-1}$ for $H^{2}$ and $s=|\varphi'(\zeta)|^{-(\alpha+2)}$ for $A^{2}_{\alpha}$.
It is not hard to see that $C_{\varphi\circ\sigma\circ\varphi}$ is not compact (see \cite[Corollary 3.14]{cm1}). Let $\tilde{\varphi}:=\varphi\circ\varphi\circ\sigma$. If $\zeta \neq \eta$, then $\overline{\tilde{\varphi}(\mathbb{D})} \subseteq \mathbb{D}$. Hence by \cite[p. 129]{cm1}, $C_{\tilde{\varphi}}$ is compact. Therefore, if $\zeta \neq \eta$, then \W\ is not quasinormal. Now assume that $\zeta =\eta$.
Since \W\ is quasinormal, \cite[Theorem 5.13]{km} shows that $\varphi\circ\sigma\circ\varphi=\varphi\circ\varphi\circ\sigma$. Since $\varphi$ is univalent, $\sigma\circ\varphi=\varphi\circ\sigma$. Hence  \cite[p. 139]{kmm} implies that $\varphi$ is a parabolic non-automorphism.\end{proof}

We can now see that \W\ is quasinormal if and only if it is normal in this case.

\begin{theorem}\label{qnn} If $\ph \in \mbox{LFT}(\mathbb{D})$ is not an automorphism of \D\ and $\psi \in \Hi$ is continuous at the Denjoy-Wolff point of \ph, then \W\ on $H^2$ or $A^2_{\alpha}$ is normal if and only if it is quasinormal. \end{theorem}
\begin{proof} Let $W_{\psi,\varphi}$ be quasinormal. First, assume as well that \ph\ has Denjoy-Wolff point $\zeta$ on $\partial \D$ and $\psi(\zeta) \neq 0$. By Proposition \ref{parabolic}, \ph\ must be a parabolic non-automorphism. Such maps converge uniformly under iteration to the Denjoy-Wolff point by Proposition \ref{uci}. By \cite[Theorem 8]{derek}, $\sigma(W_{\psi,\varphi})\subseteq \sigma(\psi(\zeta)C_{\ph})$ (\cite{derek} was written as if the setting was $H^2$, but the results hold for $A^2_\alpha$ as well, with identical arguments). In Proposition \ref{spectrum}, we showed that such a composition operator $C_\varphi$ has spectrum equal to logarithmic spiral from $1$ to $0$. Therefore, $W_{\psi,\varphi}$ has a spectrum with $1$ to $0$. However, hyponormal (and therefore also quasinormal) operators with zero area are normal \cite[Theorem 4]{stam}. If instead $\psi(\zeta) = 0$, then by \cite[Theorem 8]{derek}, the operator has spectrum equal to the singleton $\{ 0 \}$, but hyponormal operators have norm equal to their spectral radius, which is a contradiction.

Lastly, suppose $\zeta$, the Denjoy-Wolff point of \ph, is inside \D. By Theorem \ref{parabolic}, since \ph\ is not parabolic, \ph\ must not map any element of $\partial \D$ to $\partial \D$. Then, since $\overline{\varphi(\mathbb{D})} \subseteq \mathbb{D}$, $C_{\ph}$ is compact (see \cite[p. 129]{cm1}), and thus so is \W. This forces any such hyponormal operator again to be normal. The other direction is trivial. \end{proof}

In \cite{Bourdon}, these normal weighted composition operators have already been completely characterized for $H^2$ when the Denjoy-Wolff point of \ph\ is in $\D$. However, those authors do not completely identify the weights $\psi$ that make \W\ normal when \ph\ has its Denjoy-Wolff point on $\partial \D$, but only give known examples.

Next, we will turn attention to weighted composition operators with automorphic compositional symbol.

\subsection{$\ph \in \mbox{Aut}(\mathbb{D})$.}

In \cite{g}, Gunatillake showed that \W\ is invertible on $H^{2}$ if and only if $\psi$ is both bounded and bounded away from zero on $\mathbb{D}$ and $\varphi$ is an automorphism of $\mathbb{D}$. After that, in \cite{Bourdon2}, Bourdon showed the same result for invertible weighted composition operators on the weighted Bergman spaces.  In both settings, ~Hyv\"arinen et al. \cite[Corollary 5.1]{HLNS} discovered the spectrum of all such operators.

In \cite{mahsapre1}, the first two authors found all normal weighted composition operators \W\ when $\varphi \in \mbox{Aut}(\mathbb{D})$ and $\psi$ is analytic on $\overline{\mathbb{D}}$. In the following proposition, we add the assumption that  $\psi$ is bounded away from zero on $\mathbb{D}$, and by the similar idea which was used in  \cite{mahsapre1}, we characterize all invertible quasinormal weighted composition operators.

\begin{proposition}\label{auto} Suppose that \ph, not the identity and not an elliptic automorphism
of $\mathbb{D}$, is in $\mbox{Aut}(\mathbb{D})$. Let $\psi \in \Hi$ is continuous on $\partial \D$, and for each $z \in \overline{\mathbb{D}}$, $\psi(z)\neq 0$. Then \W\ is normal on $H^{2}$ or $A_{\alpha}^{2}$ if and only if $\psi(z)=\psi(0)K_{\sigma(0)}$, where $\sigma$ is the Krein adjoint of $\varphi$. \end{proposition}

\begin{proof} Let \W\ be normal. Assume that $\varphi(a)=0$, where $a \in \mathbb{D}$. Then \W\ is essentially normal. Hence \cite[Lemma 2]{Bourdon}, \cite[p. 603]{l} and
\cite[Corollary 3.5]{fash} imply that $\psi(z)=\psi(0)/(1-\overline{a}z)^{\gamma}$, where $\gamma=1$ for $H^{2}$ and $\gamma=\alpha+2$ for $A_{\alpha}^{2}$. \par
Conversely, assume that
$$w=\frac{(1-|a|^{2})^{\gamma/2}}{\psi(0)}\psi,$$
where $\gamma=1$ for $H^{2}$ and $\gamma=\alpha+2$ for $A_{\alpha}^{2}$. By \cite[Theorem 6]{Bourdon} and \cite[Corollary 3.6]{l}, we see that $W_{w,\varphi}$ is unitary. Therefore, \W\ is normal.\end{proof}

Now, we characterize the invertible quasinormal weighted composition operators on $H^2$ and $A^2_{\alpha}$.

\begin{theorem}\label{invertible} Suppose that \W\ is an invertible quasinormal weighted composition operator on $H^{2}$ or $A^{2}_{\alpha}$.
Then \W\ is normal and $\varphi$ is an automorphism; moreover, \\
(a) If $\varphi$ is the identity, then $\psi$ is a constant function. \\
(b) If $\varphi$ is an elliptic automorphism of $\mathbb{D}$, then $\psi=\psi(p)K_{p}/K_{p} \circ \varphi$, where $p \in \mathbb{D}$ is the fixed point of $\varphi$.\\
(c) If $\varphi$ is either a hyperbolic automorphism or a parabolic automorphism and $\psi \in A(\mathbb{D})$, then $\psi(z)=\psi(0)K_{\sigma(0)}$, where $\sigma$ is the Krein adjoint of $\varphi$.\end{theorem}

\begin{proof} By \cite[Theorem 3.4]{Bourdon2}, $\varphi$ is an automorphism. It is elementary to see that every invertible quasinormal operator on a Hilbert space is normal, so \W\ is normal. If $\varphi$ is the identity, then it has many fixed points in $\mathbb{D}$. Also if $\varphi$ is an elliptic automorphism, then it has only one fixed point in $\mathbb{D}$. Hence \cite[Theorem 10]{Bourdon} and \cite[Theorem 4.3]{l} imply (a) and (b). Now assume that $\varphi$ is either a hyperbolic automorphism or a parabolic automorphism and $\psi \in A(\mathbb{D})$. By \cite[Theorem 3.4]{Bourdon2}, $\psi$ is bounded away from $0$ on $\mathbb{D}$. Therefore, the result follows from Propotion \ref{auto}.\end{proof}

If $\psi$ is not bounded away from zero, then \W\ is not invertible. The spectra of such operators was studied in \cite{gaozhou}. Our conjecture is that no (non-normal) quasinormal weighted composition operators would exist in such a case.

\subsection{Unweighted composition operators.}

Here, we show our work above implies that if $\phi \in \mbox{LFT}(\D)$, then $C_\phi$ is quasinormal on $H^{2}$ or $A^{2}_{\alpha}$ if and only if it is normal and $\ph(z) = \lambda z, |\lambda| \leq 1$. In \cite{jkk}, this was proven already for $H^2$. We use a different proof technique that allows to extend the result to $A^2_{\alpha}$. \\

In \cite{z}, Zorboska showed that if $C_{\varphi}$ is hyponormal on \Ht\ or $A^2_{\alpha}$, then $\varphi(0)=0$. Since every quasinormal operator is hyponormal,  we record that result again here for our use.

\begin{proposition}[Zorboska] \label{zero}Let $C_{\varphi}$ be quasinormal on $H^{2}$ or $A^{2}_{\alpha}$. Then $\varphi(0)=0$.\end{proposition}

\begin{theorem} Let $\varphi \in \mbox{LFT}(\mathbb{D})$. Then
$C_{\varphi}$ is quasinormal on $H^{2}$ or $A^{2}_{\alpha}$ if and only if  $\ph(z)=\lambda z$, where $|\lambda| \leq 1$.\end{theorem}

\begin{proof} Suppose \C\ is quasinormal. If \ph\ is not an automorphism, then by Theorem \ref{qnn}, it is normal. Now assume that \ph\ is an automorphism. By Proposition \ref{zero}, $\ph(0)=0$, so we can easily see that $\varphi(z)=\lambda z$, where $|\lambda| = 1$, which means that \C\ is normal.
So, in all cases, \C\ is normal if it is quasinormal, and the only normal composition operators $C_{\varphi}$ on \Ht\ and $A^2_{\alpha}$ are given by $\ph(z)=\lambda z$, where $|\lambda| \leq 1$. The other direction is trivial. \end{proof}



\section{Hyponormal Weighted Composition Operators}

In this section, we focus on $\ph \in \mbox{LFT}(\D)/\mbox{Aut}(\D)$. In most scenarios, we achieve the same result  that if \W\ is hyponormal with this assumption, it is normal. However, in one case, we give new examples of hyponormal weighted composition operators which are not quasinormal. We will split this section up based on upon the location of the Denjoy-Wolff point $\zeta$ of \ph\ as well as the size of the derivative there.

\subsection{$|\zeta|=1, \ph'(\zeta) < 1$.} In this case, \ph\ is of hyperbolic type. Like the parabolic non-automorphisms, these symbols converge uniformly under iteration to the Denjoy-Wolff point (see \cite[Theorem 4]{derek}). This case has already been covered in \cite[Theorem 22]{derek} as well, but we record it again here with a simpler proof.

\begin{theorem} Suppose \ph\ is a hyperbolic non-automorphism with Denjoy-Wolff point on $\partial \D$. There is no $\psi \in \Hi$ continuous at the Denjoy-Wolff point $\zeta$ of \ph\ so that \W\ is hyponormal on \Ht\ or $A^2_{\alpha}$. \end{theorem}
\begin{proof} In \cite[Corollary 11]{derek} and \cite[Corollary 16]{derek}, it is shown that $\sigma(\W) = \sigma(\psi(\zeta) C_\phi)$ and the same is true for the point spectrum; $\sigma_{p}(\W) = \sigma_{p}(\psi(\zeta) C_\phi)$ as well (again, the work of that paper extends to $A^2_{\alpha}$ though it was written as if $H^2$ was the only possible setting).
First, suppose $\psi(\zeta) \neq 0$. The composition operator $C_\phi$ has an uncountable point spectrum \cite[Lemma 7.24]{cm1} on $H^2$, and this is also true on $A^2_{\alpha}$ since any eigenvector for \C\ on \Ht\ will also be in $A^2_{\alpha}$. Therefore \W\ has an uncountable point spectrum on each of these spaces as well. Since hyponormal operators require eigenvectors corresponding to different eigenvalues to be orthogonal \cite[Proposition 4.4]{c4} and these spaces have countable bases, we have reached a contradiction.
If instead $\psi(\zeta) = 0$, then by \cite[Theorem 8]{derek}, we have $\sigma(\W) = \{0\}$, but any hyponormal operator must have norm equal to its spectral radius, and this is a contradiction.
\end{proof}

\subsection{$|\zeta|=1, \ph'(\zeta) = 1$.} In this case, \ph\ is of parabolic type. Our earlier work shows that \W\ cannot be strictly hyponormal.

\begin{theorem}\label{parahypo} Suppose \ph\ is a parabolic non-automorphism $\psi \in \Hi$ is continuous at the Denjoy-Wolff point of $\ph$. Then \W\ is hyponormal on $H^2$ or $A^2_{\alpha}$  if and only if it is normal.
\end{theorem}
\begin{proof} The proof is contained entirely in the proof of Corollary \ref{qnn}. \end{proof}

In \cite[Theorem 20]{derek}, the following theorem was proved on $H^2$. By using the methods outlined in \cite[Theorem 20]{derek} that result holds in $A^2_{\alpha}$.

\begin{theorem}
Let $\ph:\D\rightarrow\D$ be a parabolic non-automorphism with positive translation number $t$ and Denjoy-Wolf
point $\zeta$ and let $\psi\in \Hi$ be continuous at $\zeta$.
If \W\ is hyponormal on either \Ht\ or $A^2_{\alpha}$, then it is normal and $\psi$ is a multiple
of $K_{\sigma(0)}$, where $\sigma$ is the Krein adjoint
of $\ph$. Furthermore, if $\psi(\zeta)$ is real, then \W\ is self-adjoint. \end{theorem}

It is possible that when \ph\ is a parabolic non-automorphism which is not positive, there are other weights $\psi$ so that \W\ is normal. However, we suspect that $\psi$ must be a multiple of $K_{\sigma(0)}$ if \W\ is normal and \ph\ is any parabolic non-automorphism.

\subsection{$|\zeta| < 1$, \ph\ has no fixed point on $\partial \D$.} Here, we again see that any hyponormal weighted composition operator must be normal.
\begin{theorem} Suppose $\ph \in \mbox{LFT}(\D)$ is not an automorphism, with Denjoy-Wolff point $\zeta \in \D$ and no other fixed points in $\overline{\D}$. Suppose also that $\psi \in \Hi$. Then \W\ on $H^2$ or $A^2_{\alpha}$ is hyponormal  if and only if it is normal. \end{theorem}
\begin{proof} Since \ph\ is linear fractional and does not have a fixed point on the boundary, $\overline{\ph_{n}(\D)} \subseteq \D$ for some integer $n$. Therefore, \C\ is power-compact on $H^2$ as well as on $A^2_{\alpha}$ and thus so is $\W$. Then $\sigma(\W)$ has zero area and so $\W$ is normal. The other direction is trivial. \end{proof}

These operators are already classified exactly in \cite{coko} for $H^2$.

\subsection{$|\zeta| < 1$, \ph\ has a fixed point on $\partial \D$.} A representative example from this class is $\ph(z) = z/(2-z)$. In that case, $C_{\ph}$ is known to be subnormal on $H^2$ and therefore hyponormal. Expanding on the techniques of \cite{sadraoui}, we give examples of new weighted composition operators which are hyponormal but not quasinormal, and whose weights are not necessarily linear fractional. We begin with a lemma owed to \cite{douglas}.

\begin{lemma} \cite[Theorem 1]{douglas} \label{operatorc}
A bounded operator $A$ on a Hilbert space $H$ is hyponormal if and only if there exists a bounded operator $C$ on $H$ such that $||C|| \leq 1$ and $A^{*} = CA$.
\end{lemma}

Sadraoui \cite{sadraoui} used Lemma \ref{operatorc} to good effect; the following is a narrowed example of what he proved in Section 2.5.

\begin{example}\label{sadraouiexample} For $0<s<1$, let
\begin{align*}
\ph &= sz/(1-(1-s)z), \\
\psi &= 1/(1-(1-s)z), \\
\sigma &= sz + 1-s, \\
\tau &= (sz+1-s)/(sz(1-s)+1-s+s^2), \textrm{ and } \\
\eta &= s/(sz(1-s)+1-s+s^2).
\end{align*}
Then by Proposotion \ref{cowen}, $(T_{\psi} C_{\ph})^{*} = C_\sigma$. Additionally, as proved in  \cite[Corollary 2.5.2]{sadraoui}, $C_\sigma = T_\eta C_{\tau} T_\psi C_{\ph}$ and  $\| T_\eta C_{\tau} \| = 1$. Therefore, $T_\psi C_\phi$ is hyponormal by Lemma \ref{operatorc}.\end{example}

We expand on Sadraoui's example to construct other weights $f$ so that $T_f T_\psi C_{\ph}$ is hyponormal on \Ht.

\begin{theorem}\label{sadraouidisk} Suppose that $\ph, \psi, \sigma, \eta, \tau$ are as in Example \ref{sadraouiexample}. Let $f$ be such that $f, 1/f \in H^{\infty}$. Suppose further that there exists $g \in \Hi$ such that $g \circ \sigma = f$ and $|g(z)| \leq |f(z)|$ for all $z \in \D$. Then $W_{f\psi, \ph}$ is hyponormal on $H^2$, but not quasinormal.
\end{theorem}
\begin{proof} Note that the adjoint of $W_{f\psi, \ph} = T_f T_\psi C_{\ph}$ is $C_\sigma T_f^*$. Let $ C = T_\eta C_\tau T_{g}^{*} T_{1/f}$. Since by Example 3.7, $C_{\sigma}^{\ast}=T_{\psi}C_{\varphi}$ and $C_\sigma=T_{\eta}C_{\tau}C_{\sigma}^{\ast}$, we have
\begin{align*}
C T_f T_\psi C_{\varphi} &= T_\eta C_\tau T_{g}^{*} T_{1/f} T_f T_\psi C_{\ph} \\
&= T_\eta C_\tau T_{g}^{*} T_\psi C_{\ph} \\
&= T_\eta C_\tau T_{g}^{*} C_\sigma^* \\
&= T_\eta C_\tau (C_\sigma T_{g})^{*} \\
&= T_\eta C_\tau (T_{g \circ \sigma} C_\sigma)^{*} \\
&= T_\eta C_\tau (T_{f} C_\sigma)^{*} \\
&= T_\eta C_\tau  C_\sigma^{*} T_f^* \\
&= C_\sigma T_f^*.
\end{align*}

It remains to show that $||C|| \leq 1$. Let $x \in H^{2}$ have $\|x\| \leq 1$. Note that all analytic Toeplitz operators are hyponormal on $H^2$, so $||T_{g}^{*} T_{1/f} x|| \leq ||T_{g} T_{1/f} x|| = ||T_{\frac{g}{f}}x||$. Since $|g(z)| \leq |f(z)|$ for all $z \in \D$, we have $$\left|\frac{g(z)}{f(z)}\right| \leq \left|\frac{f(z)}{f(z)}\right| = 1$$
on $\mathbb{D}$. This means that $||\frac{g}{f}||_{\infty} \leq 1$ and $||\frac{g}{f} x|| \leq ||\frac{g}{f}||_{\infty} ||x|| \leq ||x||$, so finally we have $||T_{g}^{*} T_{1/f} x|| \leq 1$. Since it has been shown that $T_\eta C_\tau$ has norm $1$, we see from the calculations here that $||Cx|| \leq 1$ for any $x$ with $||x|| \leq 1$, so $||C|| \leq 1$ and $W_{f\psi, \ph}$ is hyponormal on $H^2$. By Theorem \ref{parabolic}, it is not quasinormal.
\end{proof}

\begin{example} Let $\ph, \psi, \sigma, \eta, \tau$ be as in Example \ref{sadraouiexample} with $s = 1/2$. Then $f(z) = 2+z$ and $f(z) = e^z$ are examples of weights so that $W_{f\psi, \ph}$ is hyponormal on $H^2$ (take $g = f\circ \sigma^{-1}$ in both cases). The operator \W\ with $\psi = 2e^z/(2-z)$ and $\ph = z/(2-z)$ is the example of a hyponormal weighted composition operator such that the weight is not rational (of course, many more can be constructed from this theorem).
\end{example}

\section{Further Questions}

Below are some questions that could extend our work:

\begin{enumerate}
\item Are there quasinormal weighted composition operators on $H^2$ or $A^2_{\alpha}$ where \ph\ is an automorphism and $\psi$ is \textit{not} bounded away from zero, so that \W\ is not invertible?
\item When \ph\ is a parabolic non-automorphism, if \W\ is quasinormal, then it is normal. Some of these normal operators were already described for $H^2$ in \cite{Bourdon}, but it was assumed that $\psi$ had a very particular form. Are there other weights so that \W\ is normal? Our conjecture is no, since it is already impossible when \ph\ is a positive parabolic non-automorphism.
\item If the Denjoy-Wolff point of \ph\ is in \D\, then \ph\ must be linear fractional if \W\ is normal or even cohyponormal on $H^2$ \cite{coko}. Is this also true when \W\ is quasinormal? Furthermore, no authors have proven that if \W\ is normal on $H^2$ that \ph\ must then be linear fractional (particularly when the Denjoy-Wolff point is on the boundary of $\D$).
\item We have shown several different weights in Section 3 that make \W\ hyponormal. Is it possible to identify all weights $\psi \in \Hi$ so that \W\ is hyponormal, even with the assumption that \ph\ is linear fractional?
\end{enumerate}

\footnotesize

\bibliographystyle{amsplain}

\end{document}